%% file: wwpd_05_10_20.tex
\documentclass{article}
\usepackage{amsfonts, amsmath, amssymb, amscd, amsthm, graphicx,enumerate}
\usepackage{hyperref,xypic}
\usepackage[usenames]{color}
\usepackage{tikz}
\usepackage{calc}

\usepackage{thmtools} 
\usepackage{thm-restate}

\makeatletter
\def\blfootnote{\xdef\@thefnmark{}\@footnotetext}
\makeatother

\newtheorem{thm}{Theorem}[section]
\newtheorem{cor}[thm]{Corollary}
\newtheorem{lem}[thm]{Lemma}
\newtheorem{prop}[thm]{Proposition}

\theoremstyle{definition}
\newtheorem{defn}[thm]{Definition}
\theoremstyle{remark}
\newtheorem{rem}[thm]{Remark}

\newfont{\eufm}{eufm10}

\renewcommand{\phi}{\varphi}

\newcommand{\R}{\mathbb R}

\newcommand{\Z}{\mathbb Z}
\renewcommand{\H}{\mathbb{H}}
\newcommand{\MCG}{\operatorname{MCG}}
\newcommand{\CG}{\mathcal{C}}
\newcommand{\cl}{\mathfrak{C}}
\newcommand{\HF}{\mathcal{H}}
\newcommand{\E}{\mathcal{E}}
\newcommand{\RG}{\mathcal{R}}
\newcommand{\diam}{\operatorname{diam}}

\newcommand{\GL}{\mathcal{GL}}
\newcommand{\EL}{\mathcal{EL}}
\newcommand{\ELW}{\mathcal{ELW}}
\newcommand{\ML}{\mathcal{ML}}

\newcommand{\toCH}{\substack{CH \\ \to}}

\begin{document}

\title{WWPD elements of big mapping class groups}
\author{Alexander J. Rasmussen}

\date{\today}
\maketitle

\begin{abstract}
We study mapping class groups of infinite type surfaces with isolated punctures and their actions on the \textit{loop graphs} introduced by Bavard-Walker. We classify all of the mapping classes in these actions which are loxodromic with a WWPD action on the corresponding loop graph. The WWPD property is a weakening of Bestvina-Fujiwara's weak proper discontinuity and is useful for constructing non-trivial quasimorphisms. We use this classification to give a sufficient criterion for subgroups of big mapping class groups to have infinite-dimensional second bounded cohomology and use this criterion to give simple proofs that certain natural subgroups of big mapping class groups have infinite-dimensional second bounded cohomology.
\end{abstract}


\section{Introduction}

Mapping class groups of infinite type surfaces, commonly known as \textit{big mapping class groups}, have been a subject of intense study in the last several years. Researchers have uncovered many similarities and differences with the study of finite type mapping class groups. In this paper we consider two themes of (dis)similarity between the families of finite type mapping class groups and big mapping class groups.

\subsection{Acylindrical hyperbolicity}

Let $S$ denote a finite type or infinite type surface and let $\MCG(S)$ denote its mapping class group, \[\MCG(S)=\pi_0(\operatorname{Homeo}^+(S)).\] For the rest of the paper, all surfaces will be assumed to be orientable.

When $S$ is finite type, there is a certain infinite diameter graph $\CG(S)$, known as the curve graph acted on by $\MCG(S)$. It was first proved by Masur-Minsky in \cite{mm} that $\CG(S)$ is hyperbolic. Moreover, the action $\MCG(S)\curvearrowright \CG(S)$ has many nice properties. In particular, Bowditch proved in \cite{tight} that the action is \textit{acylindrical}, which may be thought of as a sort of proper discontinuity of the action on ``sufficiently distant pairs of points.'' This fact is enough to prove numerous results about the algebra and geometry of $\MCG(S)$ in the finite type case.

However, in \cite{notacyl}, Bavard-Genevois proved that if $S$ is infinite type then $\MCG(S)$ can \emph{never} admit an acylindrical action on a hyperbolic graph. Nonetheless, there are natural infinite diameter graphs acted on by certain big mapping class groups. In particular, if $S$ has an isolated puncture $p$, then $\MCG(S;p)$, the subgroup of $\MCG(S)$ fixing $p$, admits an action on the \textit{loop graph} $L(S;p)$ (see Section \ref{loopdefn} for the definition). This action was first studied by Bavard in \cite{ray} in a slightly different guise.

Although the action $\MCG(S;p)\curvearrowright L(S;p)$ is not acylindrical, we show that it still retains some vestige of acylindricity. However, this vestige is confined to mapping classes preserving finite type subsurfaces. To state our result, recall the notion of weakly properly discontinuous (WPD) isometries of hyperbolic metric spaces introduced by Bestvina-Fujiwara in \cite{bf}. A loxodromic element $g$ in a hyperbolic action is WPD roughly when the group action is acylindrical \textit{along the axis of $g$}. The condition WWPD was introduced by Bestvina-Bromberg-Fujiwara in \cite{bbf}. It is a weakening of WPD which allows the loxodromic element $g$ to have a large centralizer, and is useful for constructing nontrivial quasimorphisms. See Section \ref{wpdsection} for the precise definitions of WPD and WWPD.

Our main theorem may be stated as follows:

\begin{thm} \label{mainthm}
Let $\phi \in \MCG(S;p)$. Then $\phi$ is WWPD in the action of $\MCG(S;p)$ on $L(S;p)$ if and only if there exists a finite type subsurface $V\subset S$ containing $p$ such that
\begin{itemize}
\item $\phi(V)=V$,
\item $\phi|V$ is pseudo-Anosov.
\end{itemize}
\end{thm}

In particular, if $\phi\in \MCG(S;p)$ does not preserve a finite type subsurface of $S$ then it cannot be WWPD. We will sometimes say that $\phi \in \MCG(S;p)$ \textit{contains a pseudo-Anosov on a finite type witness} if there exists a subsurface $V\subset S$ as in Theorem \ref{mainthm} for $\phi$.

Any group which admits a non-elementary action on a hyperbolic space with a WPD element admits another non-elementary action on a hyperbolic space which is acylindrical \cite{osin}. Thus, by \cite{notacyl} we see that there are \textit{no} WPD elements in the action $\MCG(S;p)\curvearrowright L(S;p)$. Hence the classification of Theorem \ref{mainthm} is in some sense the best possible.

\subsection{Bounded cohomology}

One common theme in the study of big mapping class groups and finite type mapping class groups has been the study of bounded cohomology and quasimorphisms. In \cite{bf} Bestvina-Fujiwara showed that if $S$ is finite type and $\chi(S)<0$ then any subgroup of $\MCG(S)$ which is not virtually abelian has infinite-dimensional second bounded cohomology.

Bavard-Walker proved a partial extension of this result to infinite type surfaces with isolated punctures in \cite{simultaneous} (see \cite{simultaneous} Theorem 9.1.1). Their statement relies on the notion of \textit{weights} for loxodromic elements developed in that paper. In particular they prove a corollary showing that certain subgroups of $\MCG(S;p)$ with mapping classes which contain a pseudo-Anosov on a finite type witness have infinite dimensional second bounded cohomology. In Section \ref{bchsection} we prove a strengthened version of their corollary.

\begin{restatable}{cor}{bounded} \label{bounded}
Let $H\leq \MCG(S;p)$ be a subgroup such that
\begin{itemize}
\item $H$ contains a a pair of independent loxodromic elements in the action on $L(S;p)$,
\item there exists $\phi \in H$ which contains a pseudo-Anosov on a finite type witness.
\end{itemize}
Then $H_b^2(H;\R)$ is uncountably infinite dimensional.
\end{restatable}

Namely our corollary only requires the existence of \emph{one} mapping class containing a pseudo-Anosov on a finite-type witness whereas theirs requires two such mapping classes of different weights.

We use Corollary \ref{bounded} to give simple proofs that several natural subgroups of $\MCG(S;p)$ have infinite-dimensional second bounded cohomology.

In light of Theorem \ref{mainthm}, it is interesting to know that there are mapping classes $\phi\in\MCG(S;p)$ which do not preserve any finite type witness but for which there exists a quasimorphism $q:\MCG(S;p)\to \R$ with $q$ unbounded on the powers of $\phi$. See \cite{ray}.

\bigskip

\noindent{\bf Acknowledgements.} I would like to thank Yair Minsky and Javier Aramayona for valuable conversations related to this work. I would also like to thank Yair Minsky for reading a draft of this paper. I was partially supported by NSF Grant DMS-1610827.

\section{Background}

\subsection{Laminations and the coarse Hausdorff topology}

Let $V$ be a finite type surface endowed with a hyperbolic metric of finite area.

\begin{defn}
Let $\lambda$ be a geodesic lamination on $V$. We say that $\lambda$ is an ending lamination if it is
\begin{itemize}
\item minimal (every leaf of $\lambda$ is dense in $\lambda$), and
\item filling (every simple closed geodesic in  $V$ intersects $\lambda$ transversely).
\end{itemize}
\end{defn}

\begin{defn}
Let $\{\lambda_n\}_{n=1}^\infty$ be a sequence of laminations on $V$. We say that $\lambda_n$ \textit{coarse Hausdorff} converges to a lamination $\lambda$, denoted $\lambda_n \toCH \lambda$, if for any subsequence $\{\lambda_{n_i}\}_i$ such that $\lambda_{n_i}$ Hausdorff converges to a lamination $\mu$, we have $\lambda\subset \mu$.
\end{defn}

From now on we denote by:

\begin{enumerate}
\item $\ML(V)$ the space of \textit{measured laminations} on $V$ with the \textit{weak$^*$ topology},
\item $\GL(V)$ the space of \textit{geodesic laminations} on $V$ with the topology of \textit{Hausdorff convergence},
\item $\EL(V)$ the space of \textit{ending laminations} on $V$ with the topology of \textit{coarse Hausdorff convergence}.
\end{enumerate}

We also denote by $\CG(V)$ the \textit{curve graph} of $V$ which is the graph with vertex set equal to the set of homotopy classes of essential simple closed curves on $V$ and edges joining pairs of homotopy classes that may be realized disjointly. Recall that a simple closed curve is essential if it doesn't bound a disk or a once-punctured disk. The graph $\CG(V)$ was proven to be Gromov hyperbolic by Masur-Minsky in \cite{mm}.

There is another natural way to put a topology on $\EL(V)$. Namely, we may consider the subset of $\ML(V)$ consisting of measured laminations whose underlying laminations are ending laminations. Then there is an equivalence relation on this subset defined by $\lambda \sim \mu$ if the underlying laminations of $\lambda$ and $\mu$ are the same. The set of equivalence classes may be identified with $\EL(V)$ \textit{as a set} and we may endow $\EL(V)$ with the corresponding quotient topology. It was observed by Hamenst\"{a}dt in \cite{tracks} that the resulting quotient topology is in fact equal to the coarse Hausdorff topology.

With this characterization of the topology on $\EL(V)$, Klarreich's Theorem on the boundary of $\CG(V)$ is:

\begin{thm}[\cite{kla} Theorem 1.3] \label{cgboundary}
The Gromov boundary of $\CG(V)$ is equivariantly homeomorphic to $\EL(V)$.
\end{thm}

\subsection{WPD and WWPD} \label{wpdsection}

Consider a group $G$ acting by isometries on a hyperbolic metric space $X$.

\begin{defn}
We say that $g\in G$ is weakly properly discontinuous (WPD) if
\begin{itemize}
\item $g$ is loxodromic,
\item for every $x\in X$ and every $C>0$, there exists $N>0$ such that the set \[\{h\in G : d(x,h(x))\leq C, d(g^N(x),hg^N(x))\leq C\}\] is finite.
\end{itemize}
\end{defn}

Bestvina-Fujiwara proved the following result in \cite{bf}:

\begin{thm}[\cite{bf} Proposition 11] \label{WPD}
Let $V$ be finite type and $\phi\in \MCG(V)$ be pseudo-Anosov. Then $\phi$ is WPD with respect to the action of $\MCG(V)$ on $\CG(V)$.
\end{thm}

Later, in \cite{tight} Bowditch proved the stronger result that the action of $\MCG(V)$ on $\CG(V)$ is \textit{acylindrical}. However, we will not use his result in this paper.

In \cite{bbf}, Bestvina-Bromberg-Fujiwara defined a weakening of the WPD condition which is useful for applications to bounded cohomology.

\begin{defn}
We say that $g\in G$ is WWPD if
\begin{itemize}
\item $g$ is loxodromic with fixed points $g^+$ and $g^-$ on $\partial X$,
\item if $h_n\in G$ with $h_n(g^+)\to g^+$ and $h_n(g^-)\to g^-$, then there exists $N>0$ such that for all $n\geq N$ we have \[h_n(g^+)=g^+ \text{ and } h_n(g^-)=g^-.\]
\end{itemize}
\end{defn}

It is a fact that WPD elements are WWPD (see e.g. Handel-Mosher \cite{wwpd} Corollary 2.4). Hence by Theorem \ref{WPD}, every pseudo-Anosov in $\MCG(V)$ is WWPD in the action of $\MCG(V)$ on $\CG(V)$. As a corollary of this result and Klarreich's Theorem \ref{cgboundary} we have the following:

\begin{cor} \label{endinglamconv}
Let $V$ be finite type and $\phi\in \MCG(V)$ be pseudo-Anosov. Let $\lambda^+, \lambda^- \in \EL(V)$ be the fixed laminations of $\phi$. Suppose that $\psi_n \in \MCG(V)$ and $\psi_n(\lambda^+)\toCH \lambda^+$ and $\psi_n(\lambda^-) \toCH \lambda^-$ as $n\to \infty$. Then there exists $N>0$ such that for all $n\geq N$, we have \[\psi_n(\lambda^+)=\lambda^+ \text{ and } \psi_n(\lambda^-)=\lambda^- \text{ for all } n\geq N.\]
\end{cor}

\subsection{Quasimorphisms and second bounded cohomology}

Let $G$ be a group.

\begin{defn}
Let $D\geq 0$. A map $q:G\to \R$ is a \textit{quasimorphism} of defect $\leq D$ if for all $g,h\in G$, we have \[|q(gh)-q(g)-q(h)|\leq D.\] The quasimorphism $q$ is \textit{trivial} if there exists $C>0$ and a homomorphism $r:G\to \R$ with \[|r(g)-q(g)|\leq C\] for all $g\in G$.
\end{defn}

Denote by $QH(G)$ the vector space of all quasimorphisms $G\to \R$. The trivial quasimorphisms form a subspace $QH_T(G)$ of $QH(G)$ and we denote \[\widetilde{QH}(G)=QH(G)/QH_T(G)\] which we will call the \textit{space of nontrivial quasimorphisms} of $G$.

There is a well-known identification of $\widetilde{QH}(G)$ with a subspace of the \textit{second bounded cohomology} of $G$. Namely, let $H_b^2(G;\R)$ denote the second bounded cohomology of $G$ with trivial real coefficients. There is a forgetful homomorphism $H_b^2(G;\R)\to H^2(G;\R)$ and $\widetilde{QH}(G)$ is identified with the kernel of this homomorphism.

\subsection{Loop graphs} \label{loopdefn}

Let $S$ be a surface of infinite type with an isolated puncture $p$. Associated to $S$ is an infinite diameter hyperbolic graph $L(S;p)$ defined by Bavard-Walker in \cite{simultaneous}. We now recall the definition.

Denote by $E(S)$ the space of ends of $S$. An embedding $l:(0,1)\to S$ is a \textit{(simple) loop} based at $p$ if it can be continuously extended to a map $[0,1]\to S\cup E(S)$ by setting $l(0)=l(1)=p$. The \textit{loop graph} $L(S;p)$ has as a set of vertices the isotopy classes of simple loops in $S$ based at $p$. Two vertices are joined by an edge if the corresponding isotopy classes have disjoint representatives. For use later, we also recall the definition of a \textit{(simple) short ray}. An embedding $l:(0,1)\to S$ is a short ray based at $p$ if can be continuously extended to a map $[0,1]\to S\cup E(S)$ with $l(0)=p$ and $l(1)\in E(S) \setminus \{p\}$.

\begin{thm}[\cite{simultaneous} Theorem 5.4.1]
The loop graph $L(S;p)$ is hyperbolic.
\end{thm}

Denote by $\MCG(S;p)$ the subgroup of $\MCG(S)$ consisting of mapping classes which fix the puncture $p$. In particular, if $p$ is the only isolated puncture of $S$ then $\MCG(S)=\MCG(S;p)$. There is an obvious action of $\MCG(S;p)$ on $L(S;p)$ by isomorphisms.

\subsection{The boundary of the loop graph}
\label{boundarybackground}

\subsubsection{A hyperbolic metric}

Bavard-Walker identify the Gromov boundary of $L(S;p)$ with a space of rays based at $p$. To define this space of rays, it is necessary to choose a convenient hyperbolic metric on $S$. For our purposes, it is sufficient to choose a complete hyperbolic metric on $S$ of the first kind. Recall that a hyperbolic metric on a surface is said to be \textit{of the first kind} if the limit set of the corresponding Fuchsian group acting on $\H^2$ is equal to all of $\partial \H^2$.

We will frequently conflate loops and short rays with their geodesic representatives, and therefore we also assume that all loops and short rays have been put pairwise in minimal position.

Given an essential finite type subsurface $V\subset S$, we may realize the components of $\partial V$ by their geodesic representatives in $S$. The component of the complement of these geodesics containing a neighborhood of $p$ is a subsurface of $S$ homeomorphic to the interior of $V$. We will frequently conflate $V$ with this representative. Note that this representative has the property that any simple complete geodesic on $S$ which is contained in $V$ \textit{up to isotopy} is contained in this representative.

\subsubsection{The boundary}

All of the remaining material in this section is taken from Bavard-Walker \cite{boundary} and \cite{simultaneous}.

With our hyperbolic metric in hand, we may define a \textit{long ray} on $S$ to be a simple bi-infinite geodesic in $S$ which limits to $p$ at one end and does not limit to any point of $E(S)$ on the other end.

To investigate the boundary of $L(S;p)$, Bavard-Walker define the \textit{completed ray graph} $\RG(S;p)$. The vertices of $\RG(S;p)$ are the isotopy classes of all loops, short rays, and long rays on $S$. Two vertices are joined by an edge if the corresponding isotopy classes have disjoint representatives.

The graph $\RG(S;p)$ has numerous components:

\begin{thm}[\cite{simultaneous} Theorem 5.7.1]
\leavevmode
\label{components}

\begin{itemize}
\item There is a component of $\RG(S;p)$ containing all loops and short rays. This component is quasi-isometric to $L(S;p)$.
\item The other components of $\RG(S;p)$ are \textit{cliques} and each clique consists of long rays. 
\end{itemize}
\end{thm}

We call a ray \textit{high-filling} if it lies in a clique as described in the second bullet point the Theorem \ref{components}. Denote by $\HF(S;p)$ the set of high-filling rays based at $p$.

The Gromov boundary of $L(S;p)$ will be identified with a quotient of $\HF(S;p)$. However, we first have to recall the definition of the topology on $\HF(S;p)$. Bavard-Walker define the topology using an ``equator'' on $S$, which is a set of proper geodesic arcs on $S$ that cut it into two infinite-sided polygons. Since these polygons are simply-connected, one may specify rays and loops up to some amount of ambiguity by the sequence of arcs of the equator that they pass through. One may also define two rays to be close if they pass through the same arcs of the equator in the same order for a long time.

See Figure \ref{equator} for an example of an equator on the plane minus a Cantor set.

\begin{figure}[h]
\centering
\def\svgwidth{0.5\textwidth}
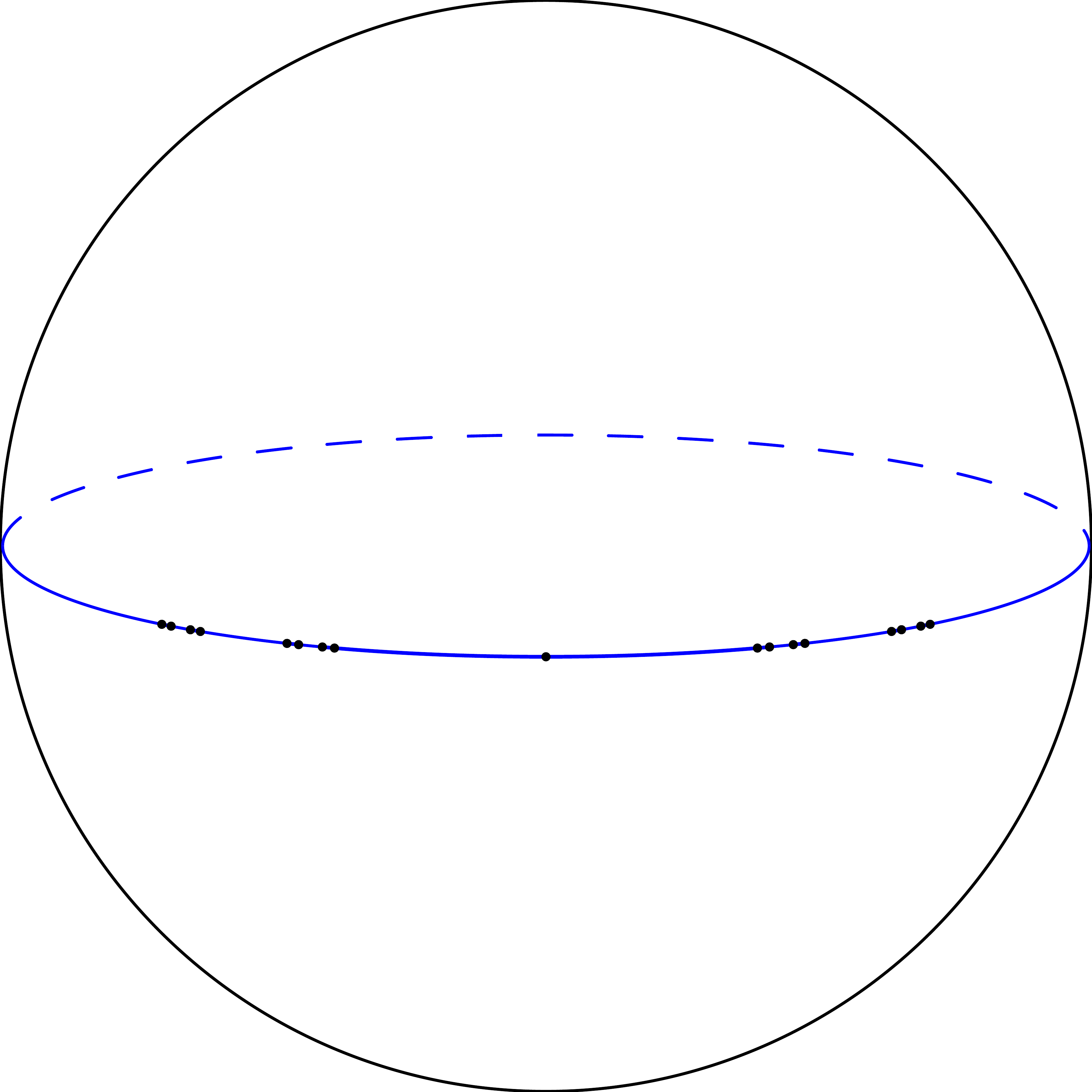

\caption{An equator for the plane minus a Cantor set is given by the countably many arcs drawn in blue.}
\label{equator}
\end{figure}

The existence of an equator is guaranteed by the following lemma:

\begin{lem}[\cite{simultaneous} Lemma 2.3.2] \label{equator}
There exists a collection $\mathcal{W}$ of mutually disjoint proper arcs in $S$ such that each end of $S$ meets at most finitely many of the arcs in $\mathcal{W}$ and the complement of $\mathcal{W}$ consists of exactly two components, each of which is simply connected.
\end{lem}

We say that two (long or short) rays or loops $l$ and $k$ \textit{$n$-begin like} each other if they cross the same first $n$ arcs of $\mathcal{W}$ in the same order (after being endowed with some orientation beginning at $p$). Now to introduce the topology on $\HF(S;p)$, we consider $l\in \HF$ and denote by $N(l,n)$ the set of high-filling rays $k\in \HF$ that $n$-begin like $l$. The sets $N(l,n)$ for $l\in \HF$ and $n\geq 0$ form a basis of open sets for a topology on $\HF$.

We may equip $\HF$ with the equivalence relation $\sim$ where $l\sim k$ if and only if $l$ and $k$ lie in the same clique of $\RG(S;p)$. The quotient space $\E(S;p)=\HF(S;p)/\sim$ inherits a quotient topology and may be identified with the set of cliques of high-filling rays in $\RG(S;p)$.

\subsubsection{The conical cover}

We wish to state finally Bavard-Walker's theorem that $\E(S;p)$ is equivariantly homeomorphic to $\partial L(S;p)$. However, there is a basic subtlety to address: $\MCG(S;p)$ admits an obvious action on $\partial L(S;p)$ but no obvious action on $\E(S;p)$. Namely, if $\phi\in \MCG(S;p)$ then $\phi$ may not be represented by any quasi-isometry of the chosen metric on $S$ and hence the image of a high-filling ray under $\phi$ may not even by a quasi-geodesic.

To overcome this difficulty, Bavard-Walker consider the \textit{conical cover} $\hat{S}$ of $S$ which is the cover of $S$ corresponding to the cyclic subgroup of $\pi_1(S)$ generated by a simple loop around $p$. The surface $\hat{S}$ has an induced hyperbolic metric which is conformal to the punctured disk. Let $\hat{p}$ denote the puncture of $\hat{S}$ and $\partial \hat{S}$ the Gromov boundary of $\hat{S}$ minus $\hat{p}$, so that $\partial \hat{S} \cong S^1$. Each short ray or loop $l$ on $S$ has a unique lift to a bi-infinite geodesic in $\hat{S}$ asymptotic to $\hat{p}$ on one end and to $\partial \hat{S}$ on the other end. Bavard-Walker show in \cite{simultaneous} that endpoints of lifts of short rays and loops are dense. Using that any homeomorphism of $S$ fixing $p$ permutes the set of short rays and loops, they show that any such homeomorphism lifts to a homeomorphism of $\hat{S}$ which \textit{extends continuously} to a homeomorphism of $\partial \hat{S}$. Moreover, this continuous extension preserves the set of endpoints of high-filling rays. This gives the desired action of $\MCG(S;p)$ on $\HF(S;p)$ by homeomorphisms and we also obtain a quotient action $\MCG(S;p)\curvearrowright \E(S;p)$.

With this action being defined, we have the following theorem of Bavard-Walker:

\begin{thm}[\cite{simultaneous} Theorem 6.6.1]
There is an equivariant homeomorphism $G:\partial L(S;p)\to \E(S;p)$.
\end{thm}

We denote $G^{-1}=F:\E\to \partial L$. We give a brief description of the homeomorphism $F$. Given a loop $c$ and a ray $l$, orient $c$ and $l$ beginning from $p$. With such orientations in hand, if $q\in c$ then we may denote by $c|(p,q]$ the unique subarc of $c$ bounded by $p$ and $q$ whose orientation matches that of $c$ when it is oriented from $p$ to $q$. Similarly, $l|(p,q]$ denotes the unique subarc of $l$ from $p$ to $q$. If $c$ and $l$ intersect at the point $q$ and the arcs $c|(p,q]$ and $l|(p,q]$ meet only at the point $q$, then $c|(p,q] \cup l|[q,p)$ is a loop, which we call a \textit{unicorn} between $c$ and $l$. If $l$ happens to be high-filling then there is an infinite path of such unicorns $P(c,l)\subset L(S;p)$ beginning at $c$. Moreover, $P(c,l)$ converges to a point of $\partial L(S;p)$. Then $F(x)$ turns out to be equal to $[P(c,l)]$ (the point on $\partial L(S;p)$ which $P(c,l)$ converges to) where $x$ is the clique containing the high-filling ray $l$. See \cite{boundary} and \cite{simultaneous} for more details.

Any loxodromic element $\phi \in \MCG(S;p)$ fixes exactly two cliques $\cl^+(\phi)$ and $\cl^-(\phi)$ in $\E$, which are attracting and repelling, respectively. Moreover we have the following result of Bavard-Walker:

\begin{thm}[\cite{simultaneous} Theorem 7.1.1]
The cliques $\cl^+(\phi)$ and $\cl^-(\phi)$ are both finite.
\end{thm}

For use later on, we also define the notion of cover convergence:

\begin{defn}
Let $l_n$ be a sequence of short rays or loops and $l$ a short ray or loop. Let $\widehat{l_n}$ be the unique lift of $l_n$ to $\hat{S}$ asymptotic to $\hat{p}$ and define $\hat{l}$ similarly. We say that $l_n$ \textit{cover converges} to $l$ if the endpoints of $\widehat{l_n}$ on $\partial \hat{S}$ converge to the endpoint of $\hat{l}$ on $\partial \hat{S}$.
\end{defn}

Note that $l_n$ cover converges to $l$ if and only if $l_n$ converges to $l$ with respect to the topology induced by the basic open sets $N(l,n)$ defined above.

\subsection{Witnesses}

We now consider subsurfaces of $S$. A simple closed curve $c\subset S$ is \textit{essential} if it doesn't bound a disk or an annulus. A subsurface $V\subset S$ is \textit{essential} if every boundary component of $V$ is essential in $S$. If $V$ is finite type then we denote by $\xi(V)$ the complexity $\xi(V)=3G-3+P+B$ where $G$ is the genus of $V$, $P$ is the number of punctures of $V$, and $B$ is the number of boundary components.

\begin{defn}
An essential subsurface $V\subset S$ is a \textit{witness} for $L(S;p)$ if every simple loop based at $p$ intersects $V$.
\end{defn}

\begin{rem}
Witnesses were first defined (in a more general context) by Masur-Schleimer in \cite{disk} where they were called ``holes.'' Schleimer has suggested that the term should be changed from ``hole'' to ``witness.''
\end{rem}

In other words $V$ is a witness if and only if it contains a neighborhood of $p$. Witnesses are interesting from the perspective of the geometry of $L(S;p)$.

\begin{prop}[\cite{afp} Corollary 4.2] \label{qiwitness}
Let $V\subset S$ be a finite type witness for $L(S;p)$ of complexity $\xi(V)\geq 2$. Then the inclusion map $L(V;p)\to L(S;p)$ is a quasi-isometric embedding.
\end{prop}

\begin{rem}
Technically this is a corollary of \cite{afp} Corollary 4.2, where we take $P=\{p\}$.
\end{rem}

Moreover, we will show below that this proposition remains true even if $\xi(V)<2$. Note that there are only a few possibilities for witnesses $V\subset S$ with $\xi(V)<2$:
\begin{itemize}
\item There are no witnesses $V$ with $\xi(V)\leq -1$.
\item If $\xi(V)=0$ then $V$ is either a torus or a thrice-punctured sphere. The first case does not occur. In the second case $L(V;p)$ consists of a single point.
\item If $\xi(V)=1$ then $V$ is either a once-punctured torus or a four-punctured sphere. The first case does not occur.
\end{itemize}

Hence the only non-trivial case of the analog of Proposition \ref{qiwitness} for $V$ a finite type witness with $\xi(V)<2$ is the case that $V$ is a four-punctured sphere.

\begin{prop}
Let $V\subset S$ be a finite type witness for $L(S;p)$. Then the inclusion map $L(V;p)\to L(S;p)$ is a quasi-isometric embedding.
\end{prop}

\begin{proof}
Define the subsurface projection $\pi_V:L(S;p)\to 2^{L(V;p)}$ as in \cite{afp}. Namely, if $c\in L(S;p)$ is contained in $V$ then we define $\pi_V(c)=\{c\}$. Otherwise there are two subarcs $c',c''$ of $c$ contained in $V$ with $c'(0)=c''(0)=p$ and $c'(1),c''(1)\in \partial V$. If $c'$ joins $p$ to the boundary component $\gamma\subset \partial V$ then we define $\pi_V(c')$ to be the boundary of a regular neighborhood of the loop $c' \cup \gamma \cup (c')^{-1}$. We define $\pi_V(c'')$ in the same way. We define \[\pi_V(c)=\{\pi_V(c'),\pi_V(c'')\}.\]

We claim that \[\diam_{L(V;p)}(\pi_V(c))\leq 2.\] As in \cite{afp} we see easily that $\pi_V(c')$ and $\pi_V(c'')$ intersect zero or two times. In the first case there is nothing to prove. So we consider two loops $a,b\in L(V;p)$ intersecting twice and show that $\diam_{L(V;p)}(a,b)\leq 2$. Orient $a$ and consider a point $q\in a\cap b$. One of the two subarcs $a|(p,q]$ or $a|[q,p)$ intersects $b$ nowhere in its interior. Suppose without loss of generality that $a|(p,q]$ intersects $b$ nowhere in its interior. Orienting $b$, we similarly have that $b|(p,q]$ intersects $a$ nowhere in its interior or $b|[q,p)$ intersects $a$ nowhere in its interior. If we suppose for example that $b|[q,p)$ intersects $a$ nowhere in its interior then we see that \[d=a|(p,q]\cup b|[q,p)\] is disjoint from both $a$ and $b$ up to isotopy. This proves the claim.

By an entirely analogous argument we have that if $c_1,c_2\in L(S;p)$ are disjoint then \[\diam_{L(V;p)}(\pi_V(c_1,c_2))\leq 2.\]

The proof of Corollary 4.2 of \cite{afp} now goes through unmodified to show that $L(V;p)\to L(S;p)$ is a quasi-isometric embedding.
\end{proof}

\section{Loop graphs of finite type witnesses}

 Denote by $\ELW(V;p)$ the space of ending laminations on witnesses $W\subset V$ for the loop graph $L(V;p)$, with the \textit{coarse Hausdorff topology}. That is, $\ELW(V;p)$ consists of minimal laminations $\lambda$ such that $\lambda$ fills an essential subsurface $W\subset V$ containing a neighborhood of $p$.
 
 If $V\subset S$ is a \textit{finite type} witness then the natural inclusion $L(V;p)\hookrightarrow L(S;p)$ is a quasi-isometric embedding. In this section we prove that $\partial L(V;p)\subset \partial L(S;p)$ may be naturally identified with $\ELW(V;p)$.

In order to do this we prove the following lemma, which describes convergence of certain sequences in the topology of $\E(S;p)$ in a way which is reminiscent of coarse Hausdorff convergence of ending laminations. The proof is basically a restatement of results from \cite{boundary} and \cite{simultaneous}.

\begin{lem} \label{topology}
Let $x$ be a \emph{finite} clique in $\E(S;p)$. Denote by $Q$ the set of endpoints on $\partial \hat{S}$ of lifts of rays $l\in x$ to rays in $\hat{S}$ beginning at $\hat{p}$. Suppose that $\{x_n\}_{n=1}^\infty$ is a sequence of cliques in $\E(S;p)$. Then the following are equivalent:
\begin{enumerate}[(1)]
\item $x_n\to x$ in $\E(S;p)$.
\item For any subsequence $\{x_{n_i}\}_{i=1}^\infty$ and choice of high-filling rays $l_{n_i}\in x_{n_i}$ such that $l_{n_i}$ cover converges to a ray or loop $l$, we have $l\in x$.
\item For any open neighborhood $\mathcal{O}\subset \partial \hat{S}$ of $Q$, there exists $N(\mathcal{O})>0$ such that if $n\geq N(\mathcal{O})$ and $l_n\in x_n$, we have $q_n\in \mathcal{O}$, where $q_n$ is the endpoint on $\partial \hat{S}$ of $\hat{l}_n$, the unique lift of $l_n$ to $\hat{S}$ beginning at $\hat{p}$.
\end{enumerate}
\end{lem}

\begin{proof}

First we show that (2) and (3) are equivalent.

To see that (2) implies (3), suppose that there exists an open neighborhood $\mathcal{O}$ of $Q$ for which there does not exist an $N(\mathcal{O})>0$ as in the statement of (3). Then we may choose a subsequence $\{x_{n_i}\}$ and $l_{n_i}\in x_{n_i}$ with the property that the unique lift $\hat{l}_{n_i}$ to $\hat{S}$ beginning at $\hat{p}$ has its other endpoint not in $\mathcal{O}$. Since $\mathcal{O}^C$ is compact, we may suppose, up to taking a further subsequence that $l_{n_i}$ cover converges to a ray or loop $l$ with $l\notin x$. Hence (2) fails.

To see that (3) implies (2) suppose that there exists a subsequence $\{x_{n_i}\}$ and $l_{n_i}\in x_{n_i}$ such that $l_{n_i}$ cover converges to a ray or loop $l$ such that $l \notin x$. Consider the unique lift $\hat{l}$ to $\hat{S}$ beginning at $\hat{p}$ and let $q$ be the endpoint of $\hat{l}$ in $\partial \hat{S}$. We may choose disjoint open neighborhoods $\mathcal{O}$ of $Q$ and $U$ of $q$. Then for all large enough $i$, we have that the endpoint of $\hat{l}_{n_i}$ in $\partial \hat{S}$ lies in $U$ and in particular not in $\mathcal{O}$, so (3) fails.

Now suppose that (1) holds. We will show that (3) also holds. Given a neighborhood $\mathcal{O}$ of $Q$, there exists $m=m(\mathcal{O})$ large enough that if $k$ $m$-begins like some $l\in x$, then the endpoint of $\hat{k}$ in $\partial \hat{S}$ lies in $\mathcal{O}$ where $\hat{k}$ is the unique lift to $\hat{S}$ beginning at $\hat{p}$ (note: this is where we use the fact that $x$ is \textit{finite}). By Proposition 6.5.2 of \cite{simultaneous}, there exists $N=N(m)>0$ such that if $n\geq N$ and $l_n \in x_n$ then $l_n$ $m$-begins like some $l \in x$ (since $F(x_n)$ lies in the neighborhood $U(F(x),R)$ specified by Proposition 6.5.2 for all $n$ large enough). This proves (3).

Finally, suppose that (2) holds. We will show that (1) also holds. If (1) did not hold, there would exist a neighborhood $U$ of $x$ in $\HF(S;p)$ and an infinite subsequence $\{x_{n_i}\}$ such that $x_{n_i}\notin U$ for all $i$. Choose $l_{n_i}\in x_{n_i}$. By compactness of $\partial \hat{S}$, we may take a further subsequence to suppose that the sequence $l_{n_i}$ cover converges to a ray or loop $l$. By assumption, $l\in x$. Fix a loop $c\in L(S;p)$. By Lemma 5.2.2 of \cite{boundary} (which holds for general surfaces $S$ with isolated punctures, see Section 6.4 of \cite{simultaneous}), we have that $x_{n_i}=G([P(c,l_{n_i})])\to G([P(c,l)])=x$. This contradicts that $x_{n_i}\notin U$ for all $i$.

\end{proof}

To begin the identification of $\partial L(V;p)$ with $\ELW(V;p)$, we first identify the image of $\partial L(V;p)$ inside $\E(S;p)$.

\begin{lem} \label{subimage}
The image $G(\partial L(V;p))$ consists of the cliques of high-filling rays in $S$ all of which are contained in $V$. Equivalently, the image consists of the cliques of high-filling rays which \emph{contain} a ray contained in $V$.
\end{lem}

\begin{proof}
First we show that every point of $G(\partial L(V;p))$ is a clique of high-filling rays, \emph{all of which} lie in $V$. Let $\{c_n\}_{n=1}^\infty \subset L(V;p)\subset L(S;p)$ be  a quasigeodesic converging to the point $[\{c_n\}]\in \partial L(V;p)\subset \partial L(S;p)$. Pass to a subsequence to assume that $\{c_n\}$ cover converges to a ray or loop $l$. The proof of Theorem 5.1.1 of \cite{boundary} (see also Theorem 6.5.3 of \cite{simultaneous}) shows that if a sequence of loops converges to a point on $\partial L(S;p)$ then any cover convergent subsequence converges to a high-filling ray. Hence $l$ is high-filling. Fix a loop $a\in L(V;p)$. Lemma 4.4.1 of \cite{boundary} (see also Section 6.2 of \cite{simultaneous} where the result is stated in full generality) may be stated as follows. If a quasigeodesic sequence of loops in $L(S;p)$ converges to a point of $\partial L(S;p)$ and if a subsequence cover converges to a long ray, then the quasigeodesic sequence stays at bounded Hausdorff distance from any unicorn path defined using that ray. Hence, $\{c_n\}$ stays at a bounded distance from $P(a,l)$ (although it may not be at bounded Hausdorff distance anymore, since we chose a cover convergent subsequence at the beginning of the proof). Since $\{c_n\}$ and $P(a,l)$ define the same point in $\partial L(V;p)$, we have that $G([\{c_n\}])=G([P(a,l)])$ is the clique of high-filling rays containing $l$ (see the description of $G$ and its inverse in Section \ref{boundarybackground}).

Realize the components of $\partial V$ by geodesics. Since no $c_n$ crosses one of these geodesics, neither does $l$. Thus $l$ is contained in $V$. The closure of $l$, minus $l$ itself, $\lambda=\overline{l}\setminus l$ is a lamination contained in $V$. We claim that $\lambda$ is minimal. To see this note first that $\overline{l}$ is itself a geodesic lamination. We claim further that $l$ is an isolated leaf of $\overline{l}$. This will prove that $\lambda$ is minimal, since in a geodesic lamination on a finite type surface, every isolated leaf either spirals onto a minimal sublamination or is asymptic to a puncture on each end (see for instance \cite{notes} Theorem 4.2.8). 

To see that $l$ is isolated in $\overline{l}$ first choose a neighborhood $U$ of $p$ such that every simple geodesic that meets $U$ must be asymptotic to $p$ (see \cite{notes} Corollary 2.2.4 or \cite{mcshane} Lemma 2 for the existence of such a neighborhood $U$). Since $l$ is not a loop it is clear that it meets $\partial U$ exactly once. Now if $l$ is not isolated in $\overline{l}$ then there exists a number $D>0$ and points $x_1,x_2,\ldots$ on $l$ converging to a point $x\in l$ such that $\operatorname{length}(l|[x_i,x_{i+1}])\geq D$ for all $i$. However, this easily yields a second point of intersection of $l$ with $U$. See Figure \ref{isolated}.

\begin{figure}[h]
\centering

\begin{tabular}{c c}
\def\svgscale{0.2}
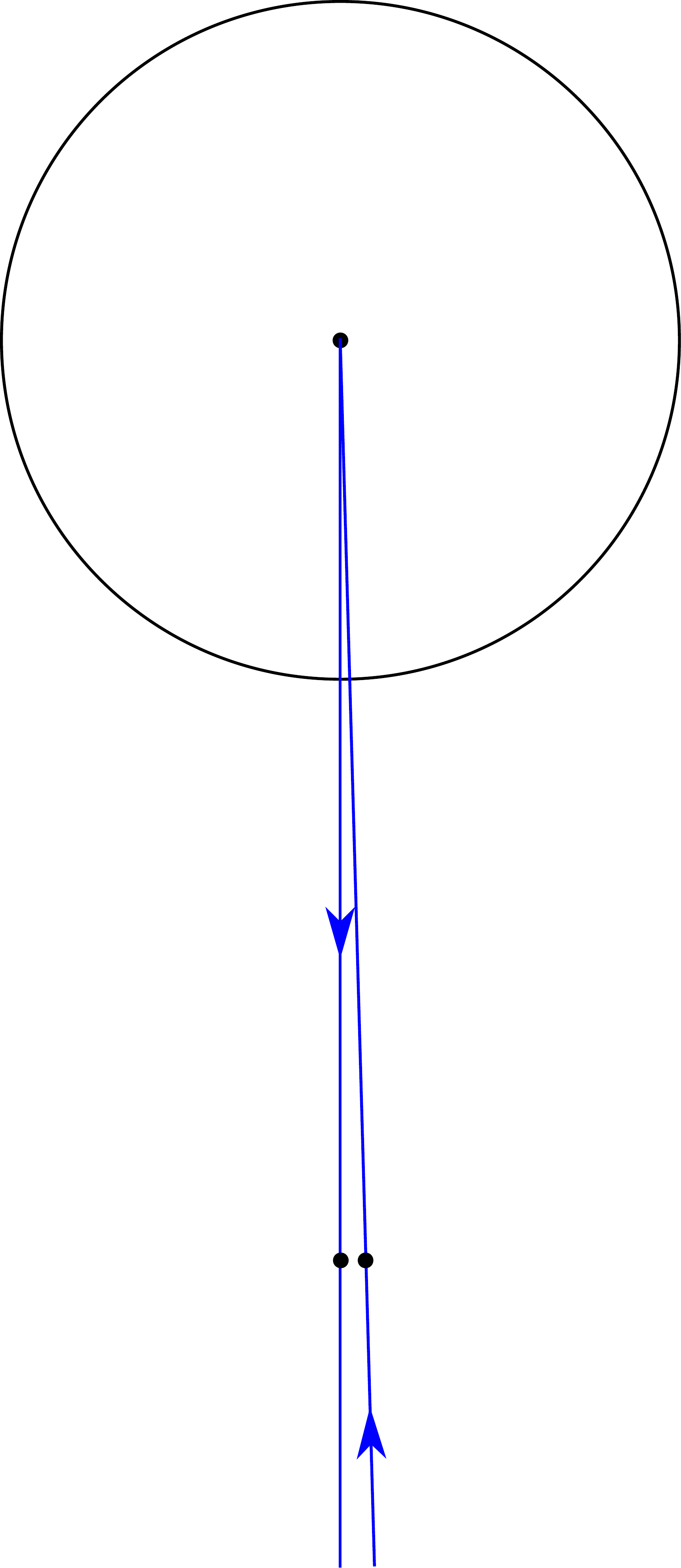 &

\def\svgscale{0.2}
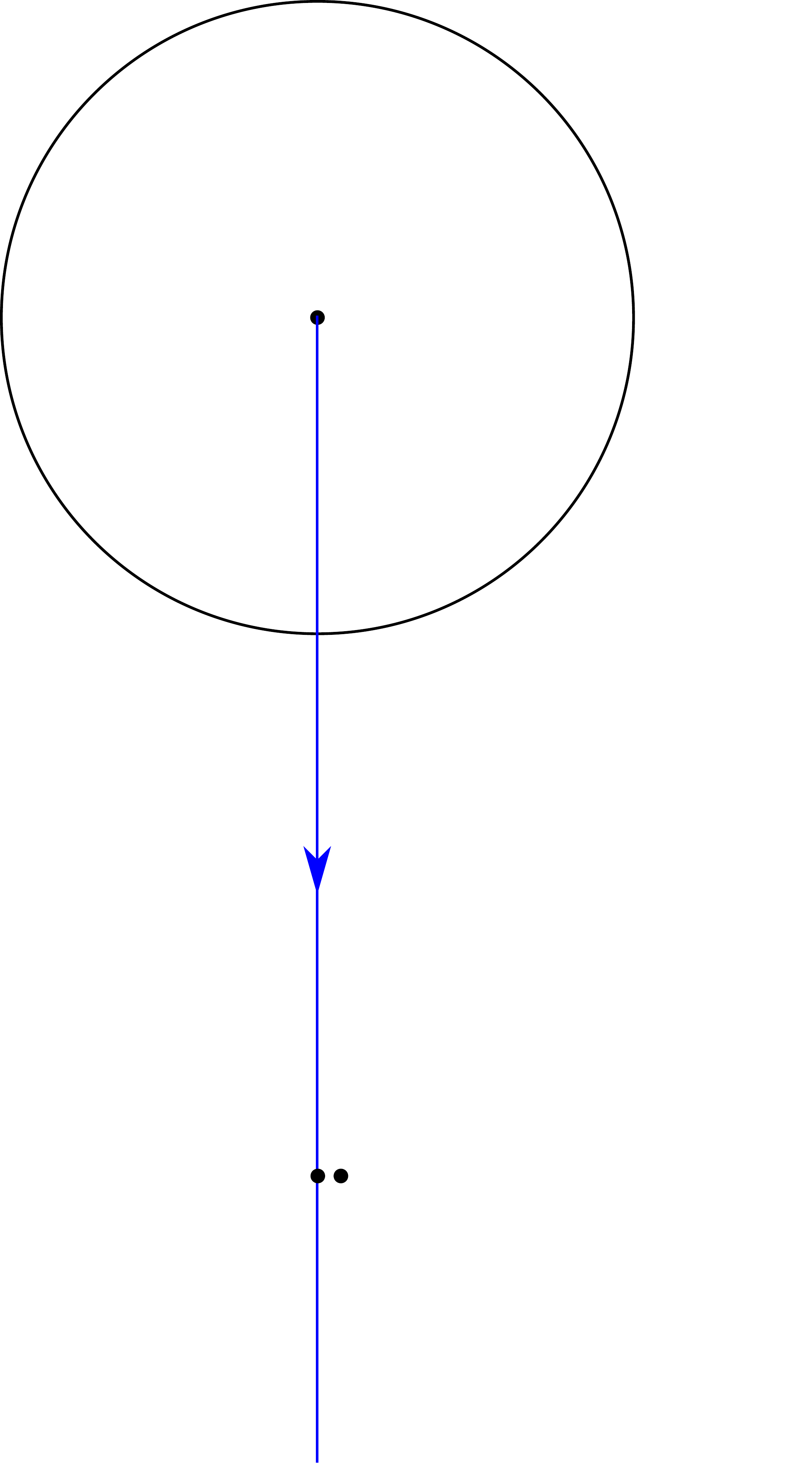

\end{tabular}
\caption{A short or long ray cannot accumulate onto itself.}
\label{isolated}
\end{figure}

Now we claim that the complementary region of $\lambda$ containing $p$ must be a once-punctured ideal polygon. For otherwise this complementary region would contain a loop $d$. Since $d$ is disjoint from $\overline{l}\setminus l$, $l$ can only intersect $d$ finitely many times. If $d\cap l=\emptyset$ then we have a contradiction to the fact that $l$ is high-filling. Otherwise, there is a point $x\in d\cap l$ and an arc $d|[x,p)$ from $x$ to $p$ such that $d|[x,p)$ intersects $l$ only at the endpoint $x$. Then $e=l|(p,x] \cup d|[x,p)$ is a loop which is disjoint from $l$ up to isotopy, again contradicting that $l$ is high-filling. 

Since the complementary region of $\lambda$ containing $p$ is a once-punctured ideal polygon, there are two possible behaviors for rays and loops $m$: either (1) $m$ intersects $\lambda$ transversely and hence also $l$ or (2) $m$ is asymptotic to an end of the ideal polygon and hence spirals onto $\lambda$.  Every ray in the clique $G([\{c_n\}])$ is of the form (2) and so spirals onto $\lambda$ and is contained in $V$.

On the other hand, we now show that every clique of $\E(S)$ which \emph{contains} a high-filling ray which is contained in $V$ is in the image $G(\partial L(V;p))$. Consider such a clique $x$ and let $l\in x$ be a ray contained in $V$. Choose also a loop $a\in L(V;p)\subset L(S;p)$. By Theorem 6.3.1 of \cite{simultaneous}, $[P(a,l)]\in \partial L(S;p)$ and $G([P(a,l)])=x$. Moreover, $P(a,l)\subset L(V;p)$ and since $L(V;p)$ is quasi-isometrically embedded, the point $[P(a,l)]$ lies in $\partial L(V;p)$. Hence $x$ is in the image $G(\partial L(V;p))$.
\end{proof}

Now we want to identify $G(\partial L(V;p))$ with $\ELW(V;p)$. We define a map $\Phi:\ELW(V;p)\to G(\partial L(V;p))$ as follows. For $\lambda \in \ELW(V;p)$, the complementary region of $\lambda$ containing $p$ is a finite-sided ideal polygon punctured at $p$. Moreover, every ray to a vertex of this polygon spirals onto $\lambda$ and therefore intersects every loop and every ray which does not spiral onto $\lambda$. Hence this set of rays is a full clique in $\RG(S;p)$ and we define $\Phi(\lambda)$ to be this clique.

The map $\Phi$ has an inverse $\Psi$ defined as follows. Consider a clique $x\in G(\partial L(V;p))$. By Lemma \ref{subimage}, each ray in $x$ is contained in $V$. Choose a ray $l\in x$. Then $\overline{l} \setminus l$ is a minimal lamination $\lambda$ contained in $V$. Moreover, because $l$ is high-filling, the complementary region of $\lambda$ containing $p$ is a once-punctured ideal polygon. So the essential subsurface $W\subset V$ filled by $\lambda$ is a witness. We define $\Psi(x)=\lambda$. To see that $\Psi$ is well-defined, by the above remarks, if $l'$ is another ray in the clique $x$, we have that $l'$ also spirals onto $\lambda$. So $\lambda$ is independent of the choice of ray in $x$.

Clearly, both $\Phi$ and $\Psi$ are equivariant.

\begin{prop} \label{homeo}
The map $\Phi$ is a homeomorphism.
\end{prop}

\begin{proof}
First, note that both $G(\partial L(V;p))$ and $\ELW(V;p)$ are metrizable. For $G(\partial L(V;p))$ this simply follows from the fact that Gromov boundaries are metrizable and $G$ is a homeomorphism. For $\ELW(V;p)$ this follows from an argument completely analogous to the argument of Klarreich in the Appendix of \cite{kla} that the ending lamination space of a finite type surface is metrizable. We omit this argument. Hence we may use sequences to test continuity of $\Phi$ and $\Psi$.

First we show that $\Phi$ is continuous. Let $\{\lambda_n\}_{n=1}^\infty\subset \ELW(V;p)$ such that $\lambda_n$ coarse Hausdorff converges to $\lambda \in \ELW(V;p)$. We wish to show that $\Phi(\lambda_n)\to \Phi(\lambda)$. If this is not the case then by Lemma \ref{topology}, there exists a subsequence $\{\lambda_{n_i}\}$ and rays $l_{n_i}$ spiraling onto $\lambda_{n_i}$ such that $l_{n_i}$ cover converges to a ray or loop $l$ which does not spiral onto $\lambda$. In particular, $l$ crosses $\lambda$. By compactness of $\GL(V)$, we may pass to a further subsequence to assume that $\lambda_{n_i}$ Hausdorff converges to a lamination $\lambda'$. By definition of coarse Hausdorff convergence, $\lambda'$ contains $\lambda$. Let $m$ be a leaf of $\lambda$ crossing $l$ at the point $q$. For any $\epsilon>0$ there exists a number $I(\epsilon)>0$ such that for all $i\geq I(\epsilon)$ there exists a leaf $m_i$ of $\lambda_{n_i}$ such that there is a subsegment of $m_i$ which is $\epsilon$-Hausdorff close to a subsegment of $m$ of radius $1/\epsilon$ centered at $q$. In particular if $\epsilon$ is small enough, if $l'$ is a ray which is $\epsilon$-close to $l$ along the segment from $p$ to $q$, then $l'$ crosses $m_i$. But by definition of cover-convergence, if $i\gg I(\epsilon)$, then $l_{n_i}$ is such a ray, so that $l_{n_i}$ crosses the leaf $m_i$ of $\lambda_{n_i}$. Since $l_{n_i}$ spirals onto $\lambda_{n_i}$, this gives a point of intersection of $l_{n_i}$ with itself, contradicting that $l_{n_i}$ is a simple ray. Thus we must have that $\Phi(\lambda_n)\to \Phi(\lambda)$. See Figure \ref{crossing}.

\begin{figure}[h]

\def\svgwidth{0.9\textwidth}
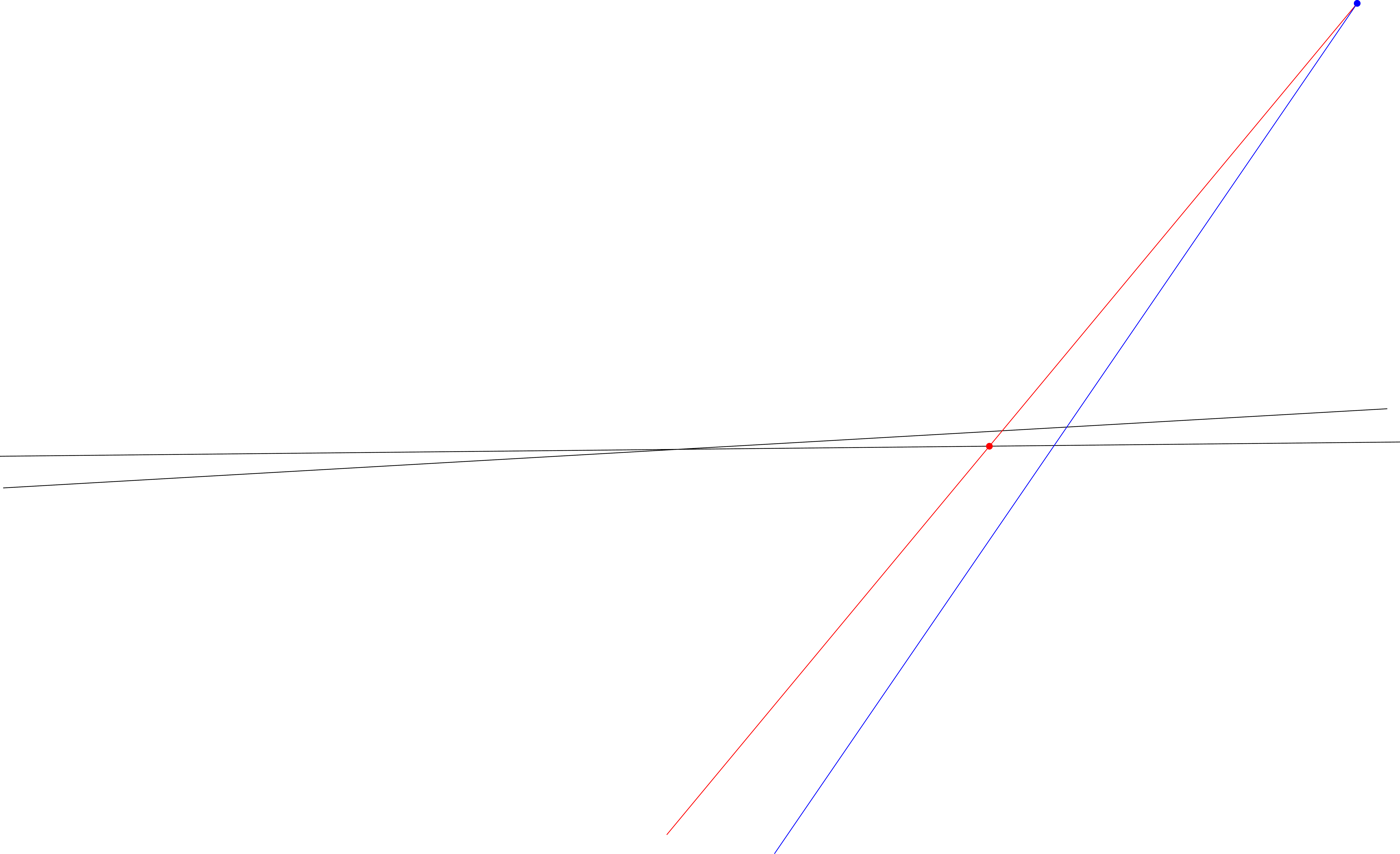

\caption{Crossings of rays in the proof of Proposition \ref{homeo}. Both blue line segments are contained in $l_{n_i}$.}
\label{crossing}
\end{figure}

Now we show that $\Psi$ is continuous. Let $\{x_n\}_{n=1}^\infty$ be a sequence of cliques of high filling rays in $V$ converging to the clique $x \in G(\partial L(V;p))$. Consider the sequence of ending laminations on witnesses $\Psi(x_n)=\lambda_n$ and the ending lamination $\Psi(x)=\lambda$. If $\lambda_n \substack{CH \\ \not \to} \lambda$ then there exists a subsequence $\{\lambda_{n_i}\}$ such that $\lambda_{n_i}$ Hausdorff converges to a lamination $\lambda'$ not containing $\lambda$. For any witness $W\subset V$, $\lambda'$ must intersect $W$ essentially. Otherwise, $\lambda_{n_i}$ is contained in $W^C$ for $i$ large enough, contradicting that $\lambda_{n_i}$ intersects every loop and short ray. In particular, taking $W$ to be the witness filled by $\lambda$, since $\lambda\not \subset \lambda'$, we have that $\lambda'$ intersects $\lambda$ transversely. Choose rays $l_{n_i}$ spiraling to $\lambda_{n_i}$. Passing to a further subsequence, we may suppose that the rays $l_{n_i}$ cover converge to a ray or loop $l$. The cliques $x_{n_i}$ converge to $x$ so by Lemma \ref{topology}, $l$ spirals to $\lambda$. This gives a point of intersection $q$ of $l$ with a leaf $m$ of $\lambda'$. As before, for any $\epsilon>0$ there exists $I(\epsilon)>0$ such that for $i\geq I(\epsilon)$, there exists a leaf $m_i$ of $\lambda_{n_i}$ with a subsegment which is $\epsilon$-Hausdorff close to a subsegment of $m$ of radius $1/\epsilon$ centered at $q$. By the definition of cover convergence, for $i\gg I(\epsilon)$, this gives an intersection of $l_{n_i}$ with the leaf $m_i$ of $\lambda_{n_i}$ and hence a point of intersection of $l_{n_i}$ with itself. This is a contradiction. The picture for this argument is the same as that of Figure \ref{crossing} with $\lambda$ replaced by $\lambda'$.

\end{proof}

From this we have the following characterization of loxodromic mapping classes acting on loop graphs of finite type surfaces:

\begin{lem} \label{loxs}
Let $V$ be a finite type surface with a puncture $p$. The mapping class $\phi\in \MCG(V;p)$ acts loxodromically on $L(V;p)$ if and only if there exists a witness $W\subset V$ for $L(V;p)$ with the property that $\phi(W)=W$ and $\phi|W$ is pseudo-Anosov.
\end{lem}

\begin{proof}
Since $\partial L(V;p)$ is equivariantly homeomorphic to $\ELW(V;p)$, if $\phi$ acts loxodromically then it fixes some $\lambda \in \ELW(V;p)$ and by definition, $\lambda$ fills a witness $W\subset V$. Therefore $\phi(W)=W$. If $\phi|W$ is not pseudo-Anosov then some power $\phi^n$ acts as the identity on $W$. Hence $\phi^n$ fixes any loop $l\subset L(V;p)$ contained in $W$. This is a contradiction.

On the other hand, it is proven in \cite{simultaneous} that if $\phi$ fixes a finite type witness $W$ and $\phi|W$ is pseudo-Anosov then $\phi$ acts loxodromically on $L(V;p)$.
\end{proof}

\section{Pseudo-Anosovs on finite type witnesses}

In this section we prove that mapping classes which contain pseudo-Anosovs on finite type witnesses are WWPD. This is one direction of Theorem \ref{mainthm}.

\begin{thm} \label{containpA}
Let $\phi\in \MCG(S;p)$ and suppose that there exists a finite type witness $V\subset S$ such that $\phi(V)=V$ and $\phi|V$ is pseudo-Anosov. Then $\phi$ is a WWPD element in the action of $\MCG(S;p)$ on $L(S;p)$.
\end{thm}

The following lemma allows us to reduce the proof of Theorem \ref{containpA} to the fact that a pseudo-Anosov on a finite-type surface is WPD with respect to the mapping class group action on the curve graph of that surface.

\begin{lem} \label{fixedsurf}
Let $V\subset S$ be a finite type witness and $\lambda^+$, $\lambda^-$ transverse ending laminations on $V$. Suppose that $\tau_n\in \MCG(S;p)$ and \[\tau_n(\Phi(\lambda^+))\to \Phi(\lambda^+) \text{ and } \tau_n(\Phi(\lambda^-))\to \Phi(\lambda^-)\] in $\E(S;p)$ as $n\to \infty$. Then there exists $N>0$ such that $\tau_n(V)=V$ for all $n\geq N$.
\end{lem}

\begin{proof}[Proof of Theorem \ref{containpA} using Lemma \ref{fixedsurf}]
Let $V$ be a finite type witness and $\phi|V$ be pseudo-Anosov. Let $\lambda^\pm \subset \EL(V)$ be the fixed laminations of $\phi$. Then since $\phi$ preserves $L(V;p)$, the fixed points of $\phi$ in $\partial L(S;p)$ are $F(\Phi(\lambda^\pm))$. Suppose that there are $\tau_n \in \MCG(S;p)$ with \[\tau_n(F(\Phi(\lambda^+)))\to F(\Phi(\lambda^+)) \text{ and }\tau_n(F(\Phi(\lambda^-)))\to F(\Phi(\lambda^-))\] in $\partial L(S;p)$. Since $F$ is an equivariant homeomorphism, this is equivalent to \[\tau_n(\Phi(\lambda^+))\to \Phi(\lambda^+) \text{ and } \tau_n(\Phi(\lambda^-))\to \Phi(\lambda^-).\] We want to show that \[\tau_n(F(\Phi(\lambda^+)))=F(\Phi(\lambda^+)) \text{ and } \tau_n(F(\Phi(\lambda^-)))=F(\Phi(\lambda^-))\] or equivalently \[\tau_n(\Phi(\lambda^+))=\Phi(\lambda^+) \text{ and } \tau_n(\Phi(\lambda^-))=\Phi(\lambda^-)\] for all sufficiently large $n$.

By Lemma \ref{fixedsurf}, there exists $N$ such that $\tau_n(V)=V$ for all $n\geq N$. Let $n\geq N$. Denote $\tau_n|V=\psi_n\in \MCG(V;p)$. We have \[\tau_n(\Phi(\lambda^\pm))=\psi_n(\Phi(\lambda^\pm))\subset G(\partial L(V;p)).\] We use the homeomorphism $\Psi:G(\partial L(V;p))\to \ELW(V;p)$. We have \[\Psi(\psi_n(\Phi(\lambda^+)))\to \Psi(\Phi(\lambda^+))=\lambda^+ \text{ and similarly } \Psi(\psi_n(\Phi(\lambda^-)))\to \lambda^-\] in the coarse Hausdorff topology on $\ELW(V;p)$. Furthermore, we have $\lambda^+\in \EL(V)$ and therefore \[\Psi(\psi_n(\Phi(\lambda^+)))=\Psi\Phi(\psi_n(\lambda^+))=\psi_n(\lambda^+)\in \EL(V)\] for all $n\geq N$ and similarly for $\lambda^-$. Hence $\psi_n(\lambda^+)$ coarse Hausdorff converges to $\lambda^+$ and $\psi_n(\lambda^-)$ coarse Hausdorff converges to $\lambda^-$. By Corollary \ref{endinglamconv} we must have $\psi_n(\lambda^+)=\lambda^+$ and $\psi_n(\lambda^-)=\lambda^-$ for all sufficiently large $n$. Therefore $\psi_n(\Phi(\lambda^+))=\Phi(\psi_n(\lambda^+))=\Phi(\lambda^+)$ and $\psi_n(\Phi(\lambda^-))=\Phi(\lambda^-)$ for all $n$ sufficiently large. This proves the theorem.

\end{proof}

\begin{proof}[Proof of Lemma \ref{fixedsurf}]
Suppose that $\{\tau_n\}\subset MCG(S;p)$ is a sequence as described in the lemma and suppose for contradiction that $\tau_n(V)\neq V$ for infinitely many $n$. Then by passing to a subsequence we may suppose that $\tau_n(V)\neq V$ for all $n$. Hence since $V$ is finite type, $\tau_n(\partial V)$ must intersect $V$ essentially for all $n$. By passing to another subsequence, we find a component $\gamma$ of $\partial V$ such that $\tau_n(\gamma)$ intersects $V$ essentially for all $n$.

We construct a useful compact subsurface $V_0\subset V$ such that $\tau_n(\gamma)\cap V_0\neq \emptyset$ for all $n$ (when $\tau_n(\gamma)$ is realized as a geodesic). This $V_0$ will have the property that any simple geodesic meeting $V_0$ intersects either $\lambda^+$ or $\lambda^-$ transversely. First of all, for each puncture $q$ of $S$ contained in $V$, there is an open neighborhood $U_q$ of $q$ with the property that any simple geodesic which meets $U_q$ must be asymptotic to $q$ (see \cite{notes} Corollary 2.2.4 or \cite{mcshane} Lemma 2). Since $V$ is finite type, we may take the finitely many neighborhoods $U_q$ to be pairwise disjoint. We remove them to form the compact subsurface $V_1\subset V$. Secondly, for each boundary component $\delta$ of $V$ there is a half open tubular neighborhood $U_\delta$ of $\delta$ in $V$ for which any simple geodesic meeting $U_\delta$ but not $V \setminus U_\delta$ must be equal to $\delta$. We may take the neighborhoods $U_\delta$ to be disjoint from each other, from the neighborhoods $U_q$ of the cusps, and from the laminations $\lambda^+$ and $\lambda^-$. We form $V_0$ by removing the neighborhoods $U_\delta$ from $V_1$.

We briefly explain why $V_0$ has the property that any simple geodesic meeting $V_0$ intersects either $\lambda^+$ or $\lambda^-$ transversely. First of all, any simple geodesic meeting $V$ besides the components of $\partial V$ either intersects $\lambda^+$ transversely or spirals onto it. However, a simple geodesic which spirals onto $\lambda^+$ intersects $\lambda^-$ transversely. Hence, since any geodesic which meets $V_0$ is not equal to a boundary component of $V$, this gives the desired property.

Now from the above description it is clear that $\tau_n(\gamma)\cap V_0\neq \emptyset$ for all $n$. By compactness of the unit tangent bundle $T_1(V_0)$, we may choose tangent vectors $(x_n,v_n)$ along $\tau_n(\gamma)$ which converge to a vector $(x,v)\in T_1(V_0)$. The geodesic $L$ through $(x,v)$ is simple and meets $V_0$. Consequently $L$ intersects either $\lambda^+$ or $\lambda^-$ transversely.

Suppose for instance that $L$ intersects $\lambda^+$ transversely. An identical argument works if $L$ intersects $\lambda^-$ transversely. Choose a ray (based at $p$) $l$ that spirals onto $\lambda^+$. We may pass to a further subsequence to obtain that $\tau_n(l)$ cover-converges to a ray or loop $m$. Since $\tau_n(l)$ cover converges to $m$, $l$ lies in the clique $\Phi(\lambda^+)$, and the cliques $\tau_n(\Phi(\lambda^+))$ converge to $\Phi(\lambda^+)$, we have by Lemma \ref{topology} that $m\in \Phi(\lambda^+)$ so $m$ must also spiral onto $\lambda^+$. Hence $m$ meets $L$ transversely. Since we have $(x_n,v_n)\to (x,v)$ and $(x_n,v_n)$ is tangent to $\tau_n(\gamma)$ for all $n$, we also have that $\tau_n(\gamma)$ meets $m$ transversely for $n$ large enough. Finally, since $\tau_n(l)$ cover converges to $m$, $\tau_n(\gamma)$ meets $\tau_n(l)$ transversely for all $n$ large enough. However, this is a contradiction since $l$ and $\gamma$ are disjoint and $\tau_n$ is a homeomorphism.
\end{proof}

\section{Mapping classes that do not contain pseudo-Anosovs on finite type witnesses}

In this section we prove the other direction of Theorem \ref{mainthm}. First we give a criterion for showing that a mapping class contains a pseudo-Anosov on a finite-type witness.

\begin{lem} \label{preservewitness}
Let $\phi\in \MCG(S;p)$ be loxodromic in the action of $\MCG(S;p)$ on $L(S;p)$. Suppose that every ray in $\cl^{\pm}(\phi)$ is contained in a finite type subsurface of $S$. Then there is a finite type witness $V\subset S$ for which $\phi|V$ is pseudo-Anosov.
\end{lem}

\begin{proof}
Consider $l\in \cl^+(\phi)$. There is a finite type subsurface $U$ containing it. The $\omega$-limit set of $l$ is a minimal geodesic lamination $\lambda$ contained in $U$ (see the argument in the proof of Lemma \ref{subimage}). Let $V\subset U$ be the essential subsurface filled by $\lambda$. We claim that $V$ contains $p$. If not, then the component of $U \setminus \lambda$ containing $p$ is a crowned hyperbolic surface $W$ such that either
\begin{itemize}
\item $\operatorname{genus}(W)>0$ or 
\item $W$ has at least two punctures.
\end{itemize}

\noindent In either case, we see that there is a nontrivial loop $c\in L(S;p)$ contained in $W$. Realize $c$ by a geodesic. We have that $c\cap l$ is finite (for this again see the proof of Lemma \ref{subimage}). However we also have $c\cap l\neq \emptyset$ since $l$ is high-filling. Orient $c$. There is a first point $y$ of $c\cap l$ with respect to the orientation on $c$. The oriented loop $c|(p,y] \cup l | [y,p)$ is disjoint from $l$ up to isotopy and is nontrivial since $c$ and $l$ are in minimal position. This contradicts that $l$ is high-filling. This proves our claim that $V$ contains $p$. 

This also proves that the component of $S \setminus \lambda$ containing $p$ is a once-punctured ideal polygon. Hence, any other ray $l'\in \cl^+(\phi)$ also spirals onto $\lambda$ (since $l'$ and $l$ are disjoint). Since $\phi$ permutes the rays in $\cl^+(\phi)$, we have $\phi(\lambda)=\lambda$ and therefore also $\phi(V)=V$. So $\phi$ preserves the finite type witness $V$.

We claim finally that $\phi$ preserves a witness $V'\subset V$ for which $\phi|V'$ is pseudo-Anosov. For $\phi$ preserves $L(V;p)\subset L(S;p)$ and moreover $L(V;p)$ is quasi-isometrically embedded. If such a $V'$ did not exist then $\phi$ would not act loxodromically on $L(V;p)$ by Lemma \ref{loxs}. But if we then choose $c\in L(V;p)\subset L(S;p)$, this implies that the orbit $\{\phi^n(c)\}_{n\in\Z}$ is not a quasi-geodesic in $L(V;p)$ and hence it is not a quasi-geodesic in $L(S;p)$. This contradicts that $\phi$ is loxodromic.
\end{proof}

Hence to finish the proof Theorem \ref{mainthm} it suffices to show that the rays in the fixed cliques of a WWPD mapping class are all contained in finite type subsurfaces of $S$. This is what we do in the proof that follows.

\begin{thm}
Suppose that $\phi\in \MCG(S;p)$ is WWPD with respect to the action of $\MCG(S;p)$ on $L(S;p)$. Then there exists a finite type witness $V\subset S$ such that $\phi(V)=V$ and $\phi|V$ is pseudo-Anosov.
\end{thm}

\begin{proof}
Let $\phi$ be WWPD. By Lemma \ref{preservewitness}, it suffices to show that every ray in $\cl^{\pm}(\phi)$ is contained in a finite type subsurface.

Suppose that $\cl^{\pm}(\phi)$ contains a ray $l$ which is not contained in a finite type subsurface of $S$. Without loss of generality, suppose that $l\in \cl^+(\phi)$. We consider an exhaustion $V_1\subset V_2 \subset \ldots $ of $S$ by finite type subsurfaces $V_n$ containing $p$. For each $n$ there is a boundary component $\gamma_n$ of $V_n$ which $l$ crosses. The sequence of rays $\{T_{\gamma_n}^i(l)\}_{i=0}^\infty$ cover converges to a ray $k$ which spirals onto $\gamma_n$. We see by analyzing the lifts $\widehat{T_{\gamma_n}^i(l)}$ to $\hat{S}$ beginning at $\hat{p}$ and their endpoints on $\partial \hat{S}$ that the elements in the sequence $\{T_{\gamma_n}^i(l)\}_{i=0}^\infty$ are distinct. In particular, there exists $i_n$ such that $T_{\gamma_n}^{i_n}(l)$ is not in the (finite) clique $\cl^+(\phi)$. Denote $\psi_n=T_{\gamma_n}^{i_n}$. The above argument implies that $\psi_n(\cl^+(\phi))\neq \cl^+(\phi)$ for any $n$. Hence, $\psi_n(F(\cl^+(\phi)))\neq F(\cl^+(\phi))$ for all $n$ and $F(\cl^+(\phi))$ is the attracting fixed point of $\phi$ on $\partial L(S;p)$.

However, we have $\psi_n(F(\cl^+(\phi)))\to F(\cl^+(\phi))$ and $\psi_n(F(\cl^-(\phi)))\to F(\cl^-(\phi))$. To see this, consider a quasi-geodesic ray $\{c_i\}_{i=1}^\infty\subset L(S;p)$ converging to $F(\cl^+(\phi))$. For any $I\geq 0$, there exists $N>0$ large enough that $c_i\subset V_N$ for all $i\leq I$. Hence \[c_i=\psi_N(c_i)=\psi_{N+1}(c_i)=\psi_{N+2}(c_i)=\ldots\] for any $i\leq I$. In other words, the quasi-geodesic rays $\{c_i\}$ and $\{\psi_n(c_i)\}$ agree up to $i=I$ for any $n\geq N$. This proves that the endpoints $\psi_n(F(\cl^+(\phi)))$ of the quasigeodesics $\{\psi_n(c_i)\}$ converge to the endpoint $F(\cl^+(\phi))$ of $\{c_i\}$. Similarly, $\psi_n(F(\cl^-(\phi)))\to F(\cl^-(\phi))$.

Thus, $\phi$ is not WWPD, by definition.
\end{proof}

\section{Quasi-morphisms and bounded cohomology} \label{bchsection}

We now turn to an application of Theorem \ref{mainthm} to bounded cohomology of subgroups of big mapping class groups.

\begin{thm}[Handel-Mosher \cite{wwpd} Theorem 2.10]
If $G\curvearrowright X$ is a hyperbolic action possessing an independent pair of loxodromic elements, and if $G$ has a WWPD element, then $H_b^2(G;\R)$ is uncountably infinite dimensional.
\end{thm}

As a corollary of this theorem we immediately obtain:

\bounded*

Understanding the quasi-morphisms and bounded cohomology of big mapping class groups is an important yet difficult goal, and could yield obstructions for groups to act by homeomorphisms on finite type surfaces (see e.g. \cite{calblog}). We hope that Corollary \ref{bounded} will be useful towards this goal. In the future, it will be necessary to better understand quasi-morphisms associated to mapping classes which \textit{do not} preserve finite type subsurfaces. Bavard has studied such quasi-morphisms in \cite{ray}.

As an application of this Corollary \ref{bounded}, we have the following result on Torelli groups. Note that this corollary, and the other remaining results of this paper most likely could also be obtained using the machinery of ``weights'' for loxodromic elements of big mapping class groups developed by Bavard-Walker. However, we find the proofs using Theorem \ref{mainthm} to be particularly quick and straightforward.

We denote by $\mathcal{I}(S)$ the \textit{Torelli} group, which is the subgroup of $\MCG(S)$ acting trivially on $H_1(S;\Z)$. We denote $\mathcal{I}(S;p)=\MCG(S;p)\cap \mathcal{I}(S)$.

\begin{cor} \label{torelli}
The second bounded cohomology $H_b^2(\mathcal{I}(S;p); \R)$ is uncountably infinite dimensional.
\end{cor}

\begin{proof}
First we show that $\mathcal{I}(S;p)$ contains a pseudo-Anosov supported on a finite type witness. Choose a separating essential simple closed curve $c$ in $S$. Choose a finite type witness $V\subset S$ containing $c$ of complexity $\geq 1$. Choose a pseudo-Anosov $\psi \in \MCG(V;p)\leq \MCG(S;p)$. For large enough $n$, $c$ and $\psi^n(c)$ fill $V$. Hence $\phi=T_c^{-1} \circ T_{\psi^n(c)}$ is a pseudo-Anosov on $V$ by Penner's criterion (see \cite{penner}). Moreover, $\phi\in \mathcal{I}(S;p)$ since any Dehn twist on a separating curve lies in $\mathcal{I}(S;p)$.

Now, to see that $\mathcal{I}(S;p)$ contains a pair of independent loxodromic elements, it suffices to consider $\phi_1\in \mathcal{I}(S;p)$ a pseudo-Anosov on a finite type witness $V_1$ and $\phi_2\in \mathcal{I}(S;p)$ a pseudo-Anosov on a finite type witness $V_2$ with $V_1\neq V_2$. The elements $\phi_1$ and $\phi_2$ are independent since they have no fixed clique in $\partial L(S;p)$ in common.
\end{proof}

As a further application, we study a construction of nontrivial quasi-morphisms on big mapping class groups. Let $S$ be a surface of finite positive genus with some finite, positive number of isolated punctures and at least one non-isolated puncture. Then the space of non-isolated punctures of $S$ is necessarily a Cantor set. Fix an isolated puncture $p$. Any mapping class of $S$ permutes the non-isolated punctures of $S$ and there is a surjective ``(non-isolated puncture)-forgetting'' homomorphism $f:\MCG(S)\to \MCG(\overline{S})$ where $\overline{S}$ is $S$ with the non-isolated punctures filled in. The puncture $p$ is identified with a puncture of $\overline{S}$ which we will also call $p$, by abuse of notation.

For any quasimorphism $q:\MCG(\overline{S};p)\to \R$, $q\circ f$ is a quasimorphism of $\MCG(S;p)$. Moreover if $q$ is nontrivial then $q\circ f$ is nontrivial by the following lemma.

\begin{rem}
This lemma is well-known to experts but we have not been able to find a proof in the literature. Hence we give a proof for completeness.
\end{rem}

\begin{lem}
Let $f:G\to H$ be a surjective homomorphism of groups and $q:H\to \R$ a nontrivial quasimorphism of $H$. Then $q\circ f$ is a nontrivial quasimorphism of $G$.
\end{lem}

\begin{proof}
Suppose that there exists a homomorphism $r:G\to \R$ with $|r(g)-q\circ f(g)|\leq D$ for all $g\in G$. In particular $r$ is bounded on $\ker(f)$ so it must in fact vanish on $\ker(f)$. So $r$ induces a homomorphism $\overline{r}:H\to \R$. That is, $r=\overline{r}\circ f$. Since $f$ is surjective, for any $h\in H$ there exists $g\in G$ with $f(g)=h$. Hence we have \[|\overline{r}(h)-q(h)|=|\overline{r}\circ f(g)-q\circ f(g)|=|r(g)-q\circ f(g)|\leq D.\] This contradicts that $q$ is nontrivial.
\end{proof}

Since $\overline{S}$ is finite type, $\MCG(\overline{S};p)$ admits an infinite dimensional space of nontrivial quasimorphisms (\cite{bf}). Thus we also immediately obtain from this fact that $\dim(\widetilde{QH}(\MCG(S;p)))=\infty$. We may ask whether every nontrivial quasimorphism of $\MCG(S;p)$ arises in this way. In other words, $f$ induces a linear map $f^*:\widetilde{QH}(\MCG(\overline{S};p))\to \widetilde{QH}(\MCG(S;p))$ and we can ask whether $f^*$ is surjective. In fact, this is not the case.

\begin{cor}
The space $H_b^2(\ker(f);\R)$ is uncountably infinite dimensional. Moreover, the induced map $f^*$ is not surjective.
\end{cor}

\begin{proof}
Consider a curve $c$ in $S$ cutting off a punctured disk containing only non-isolated punctures of $S$. Consider a finite type witness $V$ containing $c$ of complexity $\geq 1$. Choose $\psi$ a pseudo-Anosov mapping class supported on $V$. Then for large enough $n$, $c$ and $\psi^n(c)$ fill $V$ and hence $\phi=T_c^{-1}\circ T_{\psi^n(c)}$ is pseudo-Anosov by Penner's criterion. The mapping class $\phi$ is WWPD and $\phi\in \ker(f)$ since $T_c$ and $T_{\psi^n(c)}$ lie in $\ker(f)$. As in the proof of Corollary \ref{torelli}, we may construcct another loxodromic element $\phi'\in \ker(f)$ such that $\phi$ and $\phi'$ are independent. Hence by Theorem \ref{bounded}, $H_b^2(\ker(f);\R)$ is uncountably infinite dimensional.

For the proof of the second statement of the corollary, we use work of Bestvina-Bromberg-Fujiwara, see \cite{scl}. By \cite{scl} Corollary 3.2, there is a power $\phi^n$ of $\phi$ and a quasimorphism $q:\MCG(S;p)\to \R$ such that $q$ is unbounded on powers of $\phi^n$ but bounded on the powers of any elliptic element of $\MCG(S;p)$. In particular, $q$ is nontrivial. For if there were a homomorphism $r$ with $|q(g)-r(g)|\leq D$ for all $g\in \MCG(S;p)$ then $r$ would vanish on all Dehn twists in $\MCG(S;p)$ but not on large enough powers of $\phi^n$. This is a contradiction because any power of $\phi^n$ is a product of Dehn twists.

The nontrivial quasimorphism $q$ is not in the image of $f^*$. For if $\overline{q}:\MCG(\overline{S};p)\to \R$ is any quasimorphism, then $\overline{q} \circ f$ is bounded on the powers of $\phi^n$ since $f(\phi^n)=1$. In particular $\overline{q}\circ f$ is not finite $\ell^\infty$ distance from $q$ since $q$ is unbounded on the powers of $\phi^n$.
\end{proof}

\bibliographystyle{plain}
\bibliography{wwpd}

\noindent \textbf{Alexander J. Rasmussen } Department of Mathematics, Yale University, New Haven, CT 06520. \\
E-mail: \emph{alexander.rasmussen@yale.edu}

\end{document}

%% file: equator.pdf_tex
\begingroup%
  \makeatletter%
  \providecommand\color[2][]{%
    \errmessage{(Inkscape) Color is used for the text in Inkscape, but the package 'color.sty' is not loaded}%
    \renewcommand\color[2][]{}%
  }%
  \providecommand\transparent[1]{%
    \errmessage{(Inkscape) Transparency is used (non-zero) for the text in Inkscape, but the package 'transparent.sty' is not loaded}%
    \renewcommand\transparent[1]{}%
  }%
  \providecommand\rotatebox[2]{#2}%
  \newcommand*\fsize{\dimexpr\f@size pt\relax}%
  \newcommand*\lineheight[1]{\fontsize{\fsize}{#1\fsize}\selectfont}%
  \ifx\svgwidth\undefined%
    \setlength{\unitlength}{1055.03689984bp}%
    \ifx\svgscale\undefined%
      \relax%
    \else%
      \setlength{\unitlength}{\unitlength * \real{\svgscale}}%
    \fi%
  \else%
    \setlength{\unitlength}{\svgwidth}%
  \fi%
  \global\let\svgwidth\undefined%
  \global\let\svgscale\undefined%
  \makeatother%
  \begin{picture}(1,1)%
    \lineheight{1}%
    \setlength\tabcolsep{0pt}%
    \put(0,0){\includegraphics[width=\unitlength,page=1]{equator.pdf}}%
  \end{picture}%
\endgroup%

%% file: isolateda.pdf_tex
\begingroup%
  \makeatletter%
  \providecommand\color[2][]{%
    \errmessage{(Inkscape) Color is used for the text in Inkscape, but the package 'color.sty' is not loaded}%
    \renewcommand\color[2][]{}%
  }%
  \providecommand\transparent[1]{%
    \errmessage{(Inkscape) Transparency is used (non-zero) for the text in Inkscape, but the package 'transparent.sty' is not loaded}%
    \renewcommand\transparent[1]{}%
  }%
  \providecommand\rotatebox[2]{#2}%
  \newcommand*\fsize{\dimexpr\f@size pt\relax}%
  \newcommand*\lineheight[1]{\fontsize{\fsize}{#1\fsize}\selectfont}%
  \ifx\svgwidth\undefined%
    \setlength{\unitlength}{494.98313123bp}%
    \ifx\svgscale\undefined%
      \relax%
    \else%
      \setlength{\unitlength}{\unitlength * \real{\svgscale}}%
    \fi%
  \else%
    \setlength{\unitlength}{\svgwidth}%
  \fi%
  \global\let\svgwidth\undefined%
  \global\let\svgscale\undefined%
  \makeatother%
  \begin{picture}(1,2.3030591)%
    \lineheight{1}%
    \setlength\tabcolsep{0pt}%
    \put(0,0){\includegraphics[width=\unitlength,page=1]{isolateda.pdf}}%
    \put(0.36817729,0.43980916){\color[rgb]{0,0,0}\makebox(0,0)[lt]{\lineheight{1.25}\smash{\begin{tabular}[t]{l}$x$\end{tabular}}}}%
    \put(0.58333613,0.43980916){\color[rgb]{0,0,0}\makebox(0,0)[lt]{\lineheight{1.25}\smash{\begin{tabular}[t]{l}$x_i$\end{tabular}}}}%
    \put(0.40351894,0.70284531){\color[rgb]{0,0,1}\makebox(0,0)[lt]{\lineheight{1.25}\smash{\begin{tabular}[t]{l}$l$\end{tabular}}}}%
    \put(0.53030399,1.77849111){\color[rgb]{0,0,0}\makebox(0,0)[lt]{\lineheight{1.25}\smash{\begin{tabular}[t]{l}$p$\end{tabular}}}}%
    \put(0.09495269,1.83854039){\color[rgb]{0,0,0}\makebox(0,0)[lt]{\lineheight{1.25}\smash{\begin{tabular}[t]{l}$U$\end{tabular}}}}%
  \end{picture}%
\endgroup%

%% file: isolatedb.pdf_tex
\begingroup%
  \makeatletter%
  \providecommand\color[2][]{%
    \errmessage{(Inkscape) Color is used for the text in Inkscape, but the package 'color.sty' is not loaded}%
    \renewcommand\color[2][]{}%
  }%
  \providecommand\transparent[1]{%
    \errmessage{(Inkscape) Transparency is used (non-zero) for the text in Inkscape, but the package 'transparent.sty' is not loaded}%
    \renewcommand\transparent[1]{}%
  }%
  \providecommand\rotatebox[2]{#2}%
  \newcommand*\fsize{\dimexpr\f@size pt\relax}%
  \newcommand*\lineheight[1]{\fontsize{\fsize}{#1\fsize}\selectfont}%
  \ifx\svgwidth\undefined%
    \setlength{\unitlength}{633.00346314bp}%
    \ifx\svgscale\undefined%
      \relax%
    \else%
      \setlength{\unitlength}{\unitlength * \real{\svgscale}}%
    \fi%
  \else%
    \setlength{\unitlength}{\svgwidth}%
  \fi%
  \global\let\svgwidth\undefined%
  \global\let\svgscale\undefined%
  \makeatother%
  \begin{picture}(1,1.80089916)%
    \lineheight{1}%
    \setlength\tabcolsep{0pt}%
    \put(0,0){\includegraphics[width=\unitlength,page=1]{isolatedb.pdf}}%
    \put(0.28789978,0.34391299){\color[rgb]{0,0,0}\makebox(0,0)[lt]{\lineheight{1.25}\smash{\begin{tabular}[t]{l}$x$\end{tabular}}}}%
    \put(0.45614529,0.34391299){\color[rgb]{0,0,0}\makebox(0,0)[lt]{\lineheight{1.25}\smash{\begin{tabular}[t]{l}$x_i$\end{tabular}}}}%
    \put(0.31553552,0.54959663){\color[rgb]{0,0,1}\makebox(0,0)[lt]{\lineheight{1.25}\smash{\begin{tabular}[t]{l}$l$\end{tabular}}}}%
    \put(0.4146763,1.39070818){\color[rgb]{0,0,0}\makebox(0,0)[lt]{\lineheight{1.25}\smash{\begin{tabular}[t]{l}$p$\end{tabular}}}}%
    \put(0.07424918,1.4376643){\color[rgb]{0,0,0}\makebox(0,0)[lt]{\lineheight{1.25}\smash{\begin{tabular}[t]{l}$U$\end{tabular}}}}%
    \put(0,0){\includegraphics[width=\unitlength,page=2]{isolatedb.pdf}}%
  \end{picture}%
\endgroup%

%% file: crossing.pdf_tex
\begingroup%
  \makeatletter%
  \providecommand\color[2][]{%
    \errmessage{(Inkscape) Color is used for the text in Inkscape, but the package 'color.sty' is not loaded}%
    \renewcommand\color[2][]{}%
  }%
  \providecommand\transparent[1]{%
    \errmessage{(Inkscape) Transparency is used (non-zero) for the text in Inkscape, but the package 'transparent.sty' is not loaded}%
    \renewcommand\transparent[1]{}%
  }%
  \providecommand\rotatebox[2]{#2}%
  \newcommand*\fsize{\dimexpr\f@size pt\relax}%
  \newcommand*\lineheight[1]{\fontsize{\fsize}{#1\fsize}\selectfont}%
  \ifx\svgwidth\undefined%
    \setlength{\unitlength}{2678.93894946bp}%
    \ifx\svgscale\undefined%
      \relax%
    \else%
      \setlength{\unitlength}{\unitlength * \real{\svgscale}}%
    \fi%
  \else%
    \setlength{\unitlength}{\svgwidth}%
  \fi%
  \global\let\svgwidth\undefined%
  \global\let\svgscale\undefined%
  \makeatother%
  \begin{picture}(1,0.60997432)%
    \lineheight{1}%
    \setlength\tabcolsep{0pt}%
    \put(0,0){\includegraphics[width=\unitlength,page=1]{crossing.pdf}}%
    \put(0.02722483,0.2988635){\color[rgb]{0,0,0}\makebox(0,0)[lt]{\lineheight{1.25}\smash{\begin{tabular}[t]{l}$m\subset \lambda$\end{tabular}}}}%
    \put(0.0275719,0.24168261){\color[rgb]{0,0,0}\makebox(0,0)[lt]{\lineheight{1.25}\smash{\begin{tabular}[t]{l}$m_i\subset \lambda_{n_i}$\end{tabular}}}}%
    \put(0.97336617,0.59082915){\color[rgb]{0.97647059,0,0}\makebox(0,0)[lt]{\lineheight{1.25}\smash{\begin{tabular}[t]{l}$p$\end{tabular}}}}%
    \put(0,0){\includegraphics[width=\unitlength,page=2]{crossing.pdf}}%
    \put(0.91334567,0.49824652){\color[rgb]{0,0,1}\makebox(0,0)[lt]{\lineheight{1.25}\smash{\begin{tabular}[t]{l}$l_{n_i}$\end{tabular}}}}%
    \put(0.8555739,0.49811675){\color[rgb]{0.95294118,0,0}\makebox(0,0)[lt]{\lineheight{1.25}\smash{\begin{tabular}[t]{l}$l$\end{tabular}}}}%
    \put(0.71752644,0.2922619){\color[rgb]{1,0,0}\makebox(0,0)[lt]{\lineheight{1.25}\smash{\begin{tabular}[t]{l}$q$\end{tabular}}}}%
  \end{picture}%
\endgroup%